\documentclass[english]{amsart}

\usepackage{amssymb,amsmath,amsfonts,amsthm}
\usepackage{stmaryrd}
\usepackage[all,cmtip,poly]{xy}
\usepackage{svn}
\usepackage{enumerate}

\usepackage{mathrsfs}

\usepackage[latin1]{inputenc}

\usepackage{verbatim}   

\usepackage{amscd}
\usepackage{bbm} 

\usepackage{color}
\definecolor{ghcolor}{RGB}{0, 150, 200} 
\definecolor{winestain}{rgb}{0.5,0,0}
\usepackage[linktocpage=true, hidelinks, colorlinks, linkcolor=blue, citecolor=red]{hyperref} 

\swapnumbers

\newtheorem{thm}[subsubsection]{Theorem}
\newtheorem{lemma}[subsubsection]{Lemma}

\newtheorem{prop}[subsubsection]{Proposition}

\theoremstyle{definition}
\newtheorem{defn}[subsubsection]{Definition}

\theoremstyle{remark}
\newtheorem{remark}[subsubsection]{Remark}

\def\numequation{\addtocounter{subsubsection}{1}\begin{equation}}
\def\nummultline{\addtocounter{subsubsection}{1}\begin{multline}}

\newcommand{\loccit}{\textit{loc.~cit. }}

\newcommand{\hatm}{\hat{\mathfrak{M}}}
\newcommand{\hatM}{\hat{\mathfrak{M}}}

\newcommand{\barhatm}{\overline{\hat{\mathfrak{M}}}}

\newcommand{\barhatn}{\overline{\hat{\mathfrak{N}}}}

\newcommand{\barm}{\overline{\mathfrak{M}}}
\newcommand{\barn}{\overline{\mathfrak{N}}}

\newcommand{\llb}{\llbracket}  
\newcommand{\rrb}{\rrbracket}

\renewcommand{\OE}{\mathcal O_E}  

\newcommand{\barchi}{\overline{\chi}}
\newcommand{\barrho}{\overline{\rho}}

\newcommand{\Zp}{\mathbb{Z}_p}
\newcommand{\Qp}{\mathbb{Q}_p}
\newcommand{\Fp}{\mathbb{F}_p}
\newcommand{\Qpbar}{\overline{\mathbb{Q}}_p}
\newcommand{\Fpbar}{\overline{\mathbb{F}}_p}
\newcommand{\barK}{\overline{K}}

\newcommand{\Z}{\mathbb{Z}}

\newcommand{\diag}{\textnormal{diag}}
\newcommand{\col}{\textnormal{col}}

\DeclareMathOperator{\Fil}{Fil}

\DeclareMathOperator{\Gal}{Gal}
\DeclareMathOperator{\GL}{GL}

\DeclareMathOperator{\Hom}{Hom}

\DeclareMathOperator{\Ker}{Ker}
\DeclareMathOperator{\Mat}{Mat}

\DeclareMathOperator{\Mod}{Mod}

\newcommand{\HT}{\mathrm{HT}}

\newcommand{\Acris}{A_{\textnormal{cris}}}

\newcommand{\D}{\mathcal{D}}
\newcommand{\huaS}{\mathfrak{S}}
\newcommand{\huaM}{\mathfrak{M}}
\newcommand{\huam}{\mathfrak{M}}

\newcommand{\Ghat}{\hat{G}}
\newcommand{\Rhat}{\hat{\mathcal{R}}}
\newcommand{\mhat}{\hat{\huaM}}
\newcommand{\That}{\hat{T}}

\newcommand{\tn}{t^{\{n\}}}

\newcommand{\bolde}{\boldsymbol{e}}

\title{A note on crystalline liftings in the $\Qp$ case}
\author{HUI GAO}
\address{Department of Mathematics and Statistics, FIN-00014 University of Helsinki, Finland}
\email{hui.gao@helsinki.fi}

\begin{document}

\begin{abstract}
Let $p>2$ be a prime. Let $\rho$ be a crystalline representation of $G_{\Qp}$ with distinct Hodge-Tate weights in $[0, p]$, such that its reduction $\barrho$ is upper triangular. Under certain conditions, we prove that $\barrho$ has an upper triangular crystalline lift $\rho'$ such that $\HT(\rho')=\HT(\rho)$. The method is based on the author's previous work, combined with an inspiration from the work of Breuil-Herzig.

Soit $p>2$ un premier. Soit $\rho$ une repr\'esentation cristalline de $G_{\Qp}$ avec des poids distincts de Hodge-Tate dans $ [0, p] $, de telle sorte que sa r\'eduction $ \barrho $ soit triangulaire sup\'erieure. Dans certaines conditions, nous prouvons que $ \barrho $ a une \'el\'evation cristalline triangulaire sup\'erieure $ \rho'$ telle que $ \HT (\rho') = \HT (\rho) $. La m\'ethode est bas\'ee sur le travail ant\'erieur de l'auteur, combin\'e avec une inspiration de l'oeuvre de Breuil-Herzig.
 \end{abstract}

\maketitle
\pagestyle{myheadings}
\markright{A note on crystalline liftings in the $\Qp$ case}

\tableofcontents


\section{Introduction} \label{section intro}

\renewcommand{\O}{\mathcal{O}}
\newcommand{\ve}{\varepsilon} 
\newcommand{\fS}{\mathfrak{S}}

\newcommand{\gr}{{\rm gr}}
\newcommand{\fijd}{{\Fil ^{\{m_{ij }\}}\D}}

\newcommand{\mij}{\{m_{ij}\}}

\newcommand{\mfixedi}{\{m_{i,0},\ldots,m_{i,e-1}\}}


\newcommand{\mf}{\mathfrak}

\newcommand{\Aphi}{A_{\varphi}}
\newcommand{\Atau}{A_{\tau}}

\subsection{Overview}
Given (a lattice in) a crystalline representation, it is natural to study its reduction. Conversely, given a representation over an $\Fpbar$-vector space, it is natural to consider its crystalline lifts. We are particularly interested with crystalline representations, because they will have applications to weight part of Serre's conjectures (see e.g. \cite{GLS14, GLS15, Gao15unram}). In general, both these questions are notoriously difficult. For example, given an $\Fpbar$-representation, we do not even know if it has any crystalline lift. However, for applications to weight part of Serre's conjectures, we can \emph{assume} at the beginning that certain $\Fpbar$-representation already have at least one crystalline lift; the key point then is to show that it has some other \emph{nicer} crystalline lift. And this is what we do in this paper.

To state our main result, we introduce some notations first.
Let $G_{\Qp}:=\Gal(\Qpbar/\Qp)$ be the Galois group of $\Qp$. Let $E/\Qp$ be a finite extension, $\mathcal O_E$ the ring of integers, $\omega_E$ a fixed uniformizer, and $k_E=\mathcal O_E/\omega_E\mathcal O_E$ the residue field.
We will use the following notations often, (\textnormal{\textbf{CRYS}}):
\begin{itemize}
  \item Let $p>2$ be an odd prime. Let $V$ be a crystalline representation of $G_{\Qp}$ of $E$-dimension $d$, such that the Hodge-Tate weights $\HT(V)=\{ 0 = r_1 < \ldots <r_d \leq p\}$.
   \item Let $\rho=T$ be a $G_{\Qp}$-stable $\mathcal O_E$-lattice in $V$, and $\hat \huaM \in \Mod_{\huaS_{\mathcal O_E}}^{\varphi, \Ghat}$ the $(\varphi, \Ghat)$-module (with $\O_E$-coefficient) attached to $T$. Let $\barrho: =T/\omega_ET$ be the reduction. Let $\barhatm$ be the reduction of $\hatm$, and $\barm$ the reduction of $\huaM$.
\end{itemize}

\begin{thm} \label{introthm}
With notations in (\textnormal{\textbf{CRYS}}). Suppose that $\barrho$ is upper triangular, i.e., $\barrho$ is a successive extension of $d$ characters: $\barchi_1, \ldots, \barchi_d$. Suppose $\barchi_i \barchi_j^{-1} \neq \overline{\varepsilon}_p, \forall i \neq j$, where $\overline{\varepsilon}_p$ is the reduction of the cyclotomic character.
Then there exists an upper triangular crystalline representation $\rho'$ such that $\barrho' \cong \barrho$, and $\HT(\rho')=\HT(\rho)$ as sets.
\end{thm}

Theorem \ref{introthm} strengthens \cite[Cor. 0.2(1)]{Gao15unram} in the $\Qp$-case, and of course have direct application to weight part of Serre's conjectures as in \textit{loc. cit.}. In our Theorem \ref{introthm},
\begin{itemize}
  \item we do not require the Condition \textbf{(C-1)} of \cite[\S 3]{Gao15unram}, and
  \item we only require a weaker version of Condition \textbf{(C-2A)} of \cite[\S 6]{Gao15unram}.
  \item Note that Condition \textbf{(C-2B)} of \cite[\S 6]{Gao15unram} in general will never be satisfied in our current paper.
\end{itemize}
Let us also remark that Condition \textbf{(C-1)} seems to be the most difficult condition to remove in \cite{Gao15unram}.

The proof of our theorem still uses results in \cite{Gao15unram} to study the possible shape of upper triangular reductions of crystalline representations. The difference in the current paper is a different crystalline lifting technique, which is inspired by some group theory developed in \cite{BH15}. Roughly speaking, we can use the group theory to conjugate our upper triangular $\barrho$ to another upper triangular form, which can be lifted to an \emph{ordinary} (in particular, upper triangular) crystalline representation via the result of \cite{GG12}. The lifting process via \loccit is in some sense easier than those used in \cite{Gao15unram} (which is generalization of methods in \cite{GLS14, GLS15}). However, we can only apply this technique in the $\Qp$-case, because it seems that we cannot apply the group theory in \cite{BH15} to deal with general $K/\Qp$ case for our problem. Let us remark that our current paper shows a much refined structure for upper triangular reductions of crystalline representations. It is also worth pointing out that our result gives a very \emph{natural} example (see \eqref{BC}) for some of the group theories in \cite{BH15}.

The paper is organized as follows. In Section \ref{section: Kisin module review}, we review the theory of Kisin modules and $(\varphi, \hat G)$-modules with $\O_E$-coefficients.
In Section \ref{section: BH review}, we review the group theory in \cite{BH15}. In Section \ref{section: Shape study}, we study the shape of upper triangular torsion $(\varphi, \hat G)$-modules, using results in \cite{Gao15unram}, as well as techniques inspired by the group theory in Section \ref{section: BH review}. Finally in Section \ref{section: lifting}, we prove our crystalline lifting theorem.

\subsection{Notations.}
The notations in the following are taken directly from \cite{Gao15unram}. In particular, they are valid for any finite extension $K/\Qp$ (and we use $K_0$ to denote the maximal unramified sub-extension of $K$, and $k$ the residue field of $K$). See \loccit for any unfamiliar terms and more details.

In this paper, we sometimes use boldface letters (e.g., $\bolde$) to mean a sequence of objects (e.g., $\bolde=(e_1, \ldots, e_d)$ a basis of some module). We use $\Mat(?)$ to mean the set of matrices with elements in $?$. We use notations like $[u^{r_1}, \ldots, u^{r_d}]$ to mean a diagonal matrix with the diagonal elements in the bracket. We use $Id$ to mean the identity matrix. For a matrix $A$, we use $\diag A$ to mean the diagonal matrix formed by the diagonal of $A$.

In this paper, \textbf{upper triangular} always means successive extension of rank-$1$ objects.
We use notations like $\mathcal E(m_d, \ldots, m_1)$ (note the order of objects) to mean the set of all upper triangular extensions of rank-1 objects in certain categories. That is, $m$ is in $\mathcal E(m_d, \ldots, m_1)$ if there is an increasing filtration $0=\Fil^0 m \subset \Fil^1 m \subset \ldots \subset \Fil^d m =m$ such that $\Fil^i m /\Fil^{i-1}m =m_i, \forall 1 \leq i \leq d$.

We normalize the Hodge-Tate weights so that $\HT_{\kappa}(\varepsilon_p)={1}$ for any $\kappa: K \to \overline{\Qp}$, where $\varepsilon_p$ is the $p$-adic cyclotomic character.

We fix a system of elements $\{\pi_n\}_{n=0}^{\infty}$ in $\barK$, where $\pi_0=\pi$ is a uniformizer of $K$, and $\pi_{n+1}^p=\pi_n, \forall n$. Let $K_n=K(\pi_n), K_{\infty}=\cup_{n=0}^{\infty}K(\pi_n)$, and $G_{\infty}:=\Gal(\barK/K_{\infty})$.
We fix a system of elements $\{\mu_{p^n}\}_{n=0}^{\infty}$ in $\barK$, where $\mu_1=1$, $\mu_p$ is a primitive $p$-th root of unity, and $\mu_{p^{n+1}}^p=\mu_{p^n}, \forall n$.
Let
$K_{p^{\infty}} = \cup_{n=0}^{\infty} K(\mu_{p^n})$, and $\hat{K}=K_{\infty, p^{\infty}} = \cup_{n=0}^{\infty} K(\pi_n, \mu_{p^n}).$
Note that $\hat{K}$ is the Galois closure of $K_{\infty}$, and let
$\Ghat =\Gal(\hat{K}/K)$, $H_K= \Gal(\hat{K}/K_{\infty})$, and $G_{p^{\infty}} = \Gal(\hat{K}/K_{p^{\infty}}).$
When $p>2$, then $\hat G \simeq G_{p^{\infty}} \rtimes H_K$ and $G_{p^{\infty}} \simeq \Zp(1)$, and so we can (and do) fix a topological generator $\tau$ of $G_{p^{\infty}}$. And we can furthermore assume that $\mu_{p^n}=\frac{\tau(\pi_n)}{\pi_n}$ for all $n$.

Let $C=\hat{\barK}$ be the completion of $\barK$, with ring of integers $\mathcal O_C$. Let $R: = \varprojlim \mathcal O_C/p$ where the transition maps are $p$-th power map. $R$ is a valuation ring with residue field $\bar k$ ($\bar k$ is the residue field of $C$). $R$ is a perfect ring of characteristic $p$. Let $W(R)$ be the ring of Witt vectors. Let $\underline \epsilon :=(\mu_{p^n})_{n=0}^{\infty} \in R$, $\underline \pi =(\pi_n)_{n=0}^{\infty} \in R$, and let $[\underline \epsilon], [\underline \pi]$ be their Teichm\"{u}ller representatives respectively in $W(R)$. We normalize the valuation on $R$ so that $v_R(\underline \pi)=\frac{1}{e}$, where $e$ is the ramification index of $K/\Qp$.

There is a map $\theta: W(R) \to \mathcal O_C$ which is the unique universal lift of the map $R \to \mathcal O_C/p$ (projection of $R$ onto the its first factor), and $\Ker \theta$ is a principle ideal generated by $\xi= [\overline \omega]+p$, where $\overline{\omega} \in R$ with $\omega^{(0)}=-p$, and $[\overline \omega] \in W(R)$ its Teichm\"{u}ller representative. Let $B_{\rm{dR}}^{+} := \varprojlim_{n} W(R)[\frac{1}{p}]/(\xi)^n$, and $B_{\rm{dR}}:=B_{\rm{dR}}^{+}[\frac{1}{\xi}]$. Let $t: = \log([\underline{\epsilon}])$, which is an element in $B_{\rm{dR}}^{+}$.
Let $A_{\rm cris}$ denote the $p$-adic completion of the divided power envelope of $W(R)$ with respect to $\Ker(\theta)$. Let $B_{\rm cris}^{+} = A_{\rm cris}[1/p]$ and $B_{\rm cris}:= B_{\rm cris}^{+}[\frac{1}{t}]$.
The projection from $R$ to $\overline{k}$ induces a projection $\nu : W(R) \to W(\overline{k})$, since $\nu(\Ker \theta) = pW(\overline{k})$, the projection extends to $\nu: \Acris \to W(\overline{k})$, and also $\nu: B_{\text{cris}}^{+} \to W(\overline{k})[\frac{1}{p}]$. Write $I_{+}B_{\text{cris}}^{+}:=\Ker(\nu: B_{\text{cris}}^{+} \to W(\overline{k})[\frac{1}{p}]),$ and for any subring $A \subseteq B_{\text{cris}}^{+}$, write $I_{+}A = A\cap \Ker(\nu)$.

Let $\huaS: = W(k)\llb u\rrb$, $E(u)\in W(k)[u]$ the minimal polynomial of $\pi$ over $W(k)$, and $S$ the $p$-adic completion of the PD-envelope of $\huaS$ with respect to the ideal $(E(u))$.
We can embed the $W(k)$-algebra $W(k)[u]$ into $W(R)$ by mapping $u$ to $[\underline \pi]$. The embedding extends to the embeddings $\huaS \hookrightarrow S \hookrightarrow A_{\rm cris}$.

\section{Kisin modules and $(\varphi, \hat G)$-modules} \label{section: Kisin module review}
In this section, we briefly review some facts in the theory of Kisin modules and $(\varphi, \Ghat)$-modules with $\mathcal O_E$-coefficients.
The materials in this section are based on works of \cite{Kis06, Liu10, CL11, GLS14, Lev14} etc.. But here we only cite them in the form as in \cite[\S 1]{Gao15unram}, where the readers can find more detailed attributions.

\subsection{Kisin modules and $(\varphi, \hat G)$-modules with coefficients}
In this subsection, all the definitions and results are valid for any finite extension $K/\Qp$.

Recall that $\mathfrak{S}=W(k)[\![u]\!]$ with the Frobenius endomorphism $\varphi_{\huaS}: \huaS \to \huaS$ which acts on $W(k)$ via arithmetic Frobenius and sends $u$ to $u^p$. Denote
$\huaS_{\mathcal O_E}:= \huaS \otimes_{\Zp}\mathcal O_E$ and $\huaS_{k_E}:= \huaS \otimes_{\Zp}k_E = k[\![u]\!] \otimes_{\Fp} k_E.$
We can extend $\varphi_{\huaS}$ to $\huaS_{\mathcal O_E}$ (resp. $\huaS_{k_E}$) by acting on $\mathcal O_E$ (resp. $k_E$) trivially. Let $r$ be any nonnegative integer.

\begin{itemize}
 \item  Let $'\Mod_{\huaS_{\mathcal O_E}}^{\varphi}$ (called the category of Kisin modules of height $r$ with $\mathcal O_E$-coefficients) be the category whose objects are $\huaS_{\mathcal O_E}$-modules $\huaM$, equipped with $\varphi:\huaM\to\huaM$ which is a
$\varphi_{\huaS_{\mathcal O_E}}$-semi-linear morphism such that the span of $\text{Im}(\varphi)$ contains $E(u)^{r}\huaM$.
The morphisms in the category are $\huaS_{\mathcal O_E}$-linear maps that commute with $\varphi$.

\item Let $\Mod_{\huaS_{\mathcal O_E}}^{\varphi}$ be the full subcategory of $'\Mod_{\huaS_{\mathcal O_E}}^{\varphi}$ with $\huaM \simeq \oplus_{i \in I}\huaS_{\mathcal O_E}$ where $I$ is a finite set.
 Let $\Mod_{\huaS_{k_E}}^{\varphi}$ be the full subcategory of $'\Mod_{\huaS_{\mathcal O_E}}^{\varphi}$ with $\huaM \simeq \oplus_{i \in I}\huaS_{k_E}$ where $I$ is a finite set.
\end{itemize}

For any integer $n \geq 0$, write $n =(p-1)q(n)+r(n)$ with $q(n)$ and $r(n)$ the quotient and residue of $n$ divided by $p-1$. Let $\tn=(p^{q(n)}\cdot q(n)!)^{-1}\cdot t^n$, we have $\tn \in A_{\text{cris}}$.
We define a subring of $B_{\text{cris}}^{+}$,
$\mathcal{R}_{K_0} :=\left\{ \sum_{i=0}^{\infty} f_i t^{\{i\}} , f_i \in S_{K_0}, f_i \to 0 \text{ as } i \to \infty \right\}.$
Define $\hat{\mathcal{R}}:= \mathcal{R}_{K_0} \cap W(R)$. Then $\Rhat$ is a $\varphi$-stable subring of $W(R)$, which is also $G_K$-stable, and the $G_K$-action factors through $\Ghat$. Denote $\Rhat_{\mathcal O_E}:= \Rhat \otimes_{\Zp}\mathcal O_E$, $W(R)_{\mathcal O_E}:= W(R) \otimes_{\Zp}\mathcal O_E,$
and extend the $G_K$-action and $\varphi$-action on them by acting on $\mathcal O_E$ trivially.
Note that $\huaS_{\mathcal O_E} \subset \hat R_{\mathcal O_E}$, and let $\varphi: \huaS_{\mathcal O_E} \to \hat R_{\mathcal O_E}$ be the composite of $\varphi_{\huaS_{\mathcal O_E}}: \huaS_{\mathcal O_E} \to \huaS_{\mathcal O_E}$ and the embedding $\huaS_{\mathcal O_E} \to \hat R_{\mathcal O_E}$.

\begin{defn}
Let $'\Mod_{\huaS_{\mathcal O_E}}^{\varphi, \Ghat}$ be the category (called the category of $(\varphi, \Ghat)$-modules of height $r$ with $\mathcal O_E$-coefficients) consisting of triples $(\huaM, \varphi_{\huaM}, \Ghat)$ where,
\begin{enumerate}
\item $(\huaM, \varphi_{\huaM}) \in '\Mod_{\huaS_{\mathcal O_E}}^{\varphi}$ is a Kisin module of height $r$;
\item $\Ghat$ is a $\Rhat_{\mathcal O_E}$-semi-linear $\Ghat$-action on $\mhat := \Rhat_{\mathcal O_E} \otimes_{\varphi, \huaS_{\mathcal O_E}} \huaM$;
\item $\Ghat$ commutes with $\varphi_{\mhat} : =\varphi_{\Rhat_{\mathcal O_E}}\otimes \varphi_{\huaM}$;
\item Regarding $\huaM$ as a $\varphi(\huaS_{\mathcal O_E})$-submodule of $\mhat$, then $\huaM \subseteq \mhat^{H_K}$;
\item $\Ghat$ acts on the $\mhat/(I_{+}\hat{R})\mhat$ trivially.
\end{enumerate}
A morphism between two $(\varphi, \Ghat)$-modules is a morphism in $\Mod_{\huaS_{\mathcal O_E}}^{\varphi}$ which commutes with $\Ghat$-actions.
\end{defn}

We denote $\Mod_{\huaS_{\mathcal O_E}}^{\varphi, \Ghat}$ to be the full subcategory of $'\Mod_{\huaS_{\mathcal O_E}}^{\varphi, \Ghat}$ where $\huaM \in \Mod_{\huaS_{\mathcal O_E}}^{\varphi}$;
and we denote $\Mod_{\huaS_{k_E}}^{\varphi, \Ghat}$ for the full subcategory of $'\Mod_{\huaS_{\mathcal O_E}}^{\varphi, \Ghat}$ where $\huaM \in \Mod_{\huaS_{k_E}}^{\varphi}$.

We can associate representations to $(\varphi, \Ghat)$-modules.

\begin{thm}\cite[Thm. 1.2, Thm. 1.4]{Gao15unram}  \hfill
\begin{enumerate}
\item Suppose $\mhat \in \Mod_{\huaS_{\mathcal O_E}}^{\varphi, \Ghat}$ where $\huaM$ is of $\huaS_{\mathcal O_E}$-rank $d$, then
      $$ \That(\mhat) := \Hom_{\Rhat, \varphi} (\mhat, W(R))$$
      is a finite free $\mathcal O_E$-representation of $G_K$ of rank $d$.

    \item Suppose $\mhat \in \Mod_{\huaS_{k_E}}^{\varphi, \Ghat}$ where $\huaM$ is of $\huaS_{k_E}$-rank $d$, then
      $$ \That(\mhat) := \Hom_{\Rhat, \varphi} (\mhat, W(R)\otimes_{\Zp}\Qp/\Zp)$$
      is a finite free $k_E$-representation of $G_K$ of dimension $d$.

     \item   For $\mhat \in \Mod_{\huaS_{\mathcal O_E}}^{\varphi, \Ghat}$, we have $\That(\mhat/\omega_E\mhat) \simeq \That(\mhat)/\omega_E\That(\mhat)$.
\end{enumerate}
\end{thm}

When $p>2$, the theory of $(\varphi, \Ghat)$-modules becomes simpler.
\begin{lemma}\cite[Lem. 1.6]{Gao15unram} \label{pnot2}
Suppose $p>2$. Let $\mhat \in \Mod_{\huaS_{\mathcal O_E}}^{\varphi, \Ghat}$. Then $\mhat$ is uniquely determined up to isomorphism by the following information:
\begin{enumerate}
  \item A matrix $A_{\varphi} \in \Mat(\huaS_{\mathcal O_E})$ for the Frobenius $\varphi: \huaM \to \huaM$, such that there exist $B \in \Mat(\huaS_{\mathcal O_E})$ with $A_{\varphi}B=E(u)^r Id$.
   \item A matrix $A_{\tau} \in \Mat(\hat{R}_{\mathcal O_E})$ (for the $\tau$-action $\tau: \mhat \to \mhat$) such that
   \begin{itemize}
     \item $A_{\tau}-Id \in \Mat(I_{+}\Rhat_{\mathcal O_E}),$
     \item $A_{\tau}\tau(\varphi(A_{\varphi}))=\varphi(A_{\varphi})\varphi(A_{\tau}).$
     \item $g(A_{\tau}) = \prod_{k=0}^{\varepsilon_p(g)-1} \tau^k(A_{\tau})$ for all $g\in G_{\infty}$ such that $\varepsilon_p(g) \in \mathbb{Z}^{\geq 0}$.
   \end{itemize}
\end{enumerate}
For $\barhatm \in \Mod_{\huaS_{k_E}}^{\varphi, \Ghat}$, it is also uniquely determined up to isomorphism by its matrix $A_{\varphi}$ and $A_{\tau}$ satisfying similar conditions as above.
\end{lemma}

\subsection{Rank 1 Kisin modules and $(\varphi, \Ghat)$-modules}
We only recall the following definitions and results in the $\Qp$ case.
\begin{defn}\hfill
\begin{enumerate}
  \item Suppose $t$ is a non-negative integer, $a \in k_{E}^{\times}$. Let $\barm(t; a)$  be the rank-$1$ module in $\Mod_{\huaS_{k_E}}^{\varphi}$ such that $\barm(t; a)$ is generated by some basis $e$, and $\varphi(e)=au^t e$.

  \item  Suppose $t$ is a non-negative integer, $\hat a \in \mathcal O_{E}^{\times}$. Let $\huam(t; \hat a)$  be the rank-$1$ module in $\Mod_{\huaS_{\OE}}^{\varphi}$  such that $\huam(t; \hat a)$ is generated by some basis $\tilde e$, and $\varphi(\tilde e)=\hat{a}(u-p)^t \tilde e$.
\end{enumerate}
\end{defn}

\begin{lemma}\cite[Lem. 1.11]{Gao15unram} \label{lemma: rank 1}\hfill
\begin{enumerate}
  \item Any rank $1$ module in $\Mod_{\huaS_{k_E}}^{\varphi}$ is of the form $\barm(t; a)$ for some $t$ and $a$.
    \item When $\hat a$ is a lift of $a$, $\huaM(t; \hat a)/\omega_E\huaM(t; \hat a) \simeq \barm(t; a)$.
    \item There is a unique $\hatm(t; \hat a) \in \Mod_{\huaS_{\mathcal O_E}}^{\varphi, \Ghat}$ such that the ambient Kisin module of $\hatm(t; \hat a)$ is $\huaM(t; \hat a)$, and $\That(\hatm(t; \hat a))$ is a crystalline character.
    In fact, $\That(\hatm(t; \hat a))= \lambda_{\hat a}  \psi^{t},$ where $\psi$ is a certain crystalline character such that $\HT(\psi)={1}$, and $\lambda_{\hat a}$ is the unramified character of $G_{\Qp}$ which sends the arithmetic Frobenius to $\hat a$.

    \item There is a unique $\overline{\hatm}(t; a) \in \Mod_{\huaS_{k_E}}^{\varphi, \Ghat}$ such that the ambient Kisin module is
    $\barm(t; a)$. Furthermore, $\That(\overline{\hatM}(t; a))$ is the reduction of $\That(\hat \huaM(t; \hat a))$ for any lift $\hat a \in \mathcal O_E$ of $a$.
\end{enumerate}
\end{lemma}


\section{Some group theory} \label{section: BH review}

\newcommand{\alg}{\textnormal{alg}}
\newcommand{\Rplus}{R^{+}}
\newcommand{\ee}{\epsilon}

We recall some group theory, which will be useful for our work. All the materials in this section are developed in \cite[\S 2.3]{BH15}, for general split connected reductive groups. But we will only need it for $\GL_d$, which we recall.

Let $H$ be the algebraic group ${\GL}_{d}$, $T$ the torus consisting of diagonal matrices, $B$ the Borel consisting of upper triangular matrices, and $U$ the unipotent radical consisting of unipotent matrices.

We have $X(T): =\Hom_{\alg}(T, \mathbb G_m)=\Z \ee_1\oplus \cdots \oplus \Z \ee_d$, where $\ee_i$ is the character sending the diagonal matrix $[x_1,\ldots, x_d]$ to $x_i$. Let $S=\{\ee_i-\ee_{i+1} : 1\leq i\leq d-1\}$ be the simple roots, and let $R^{+}=\{\ee_i-\ee_j: 1\leq i <j \leq d\}$ be the positive roots. Denote $W$ the Weyl group of $H$, which is isomorphic to the permutation group $S_d$.
If $\alpha=\ee_i -\ee_j \in \Rplus$, let $U_{\alpha} \subset H$ be the root subgroup, which corresponds to the unipotent upper triangular matrices where the only nonzero element above the diagonal is at the $(i, j)$-position.

\begin{defn}
A subset $C \subseteq \Rplus$ is called \emph{closed} if the following condition is satisfied: if $\alpha \in C, \beta \in C$ and $\alpha+\beta \in \Rplus$, then $\alpha+\beta \in C$.
\end{defn}

For a closed subset $C \subseteq \Rplus$, let $U_C \subseteq U$ be the Zariski closed subgroup of $B$ generated by the subgroups $U_{\alpha}$ for all $\alpha \in C$. Let $B_C=TU_C \subseteq B$.
If $C= \{ \ee_{i_1}-\ee_{j_1}, \ldots, \ee_{i_m}-\ee_{j_m} \}$ is a closed subset of $\Rplus$, then it is easy to see that $B_C$ corresponds to the matrices where the only nonzero elements above the diagonal are at the positions $(i_\ell, j_\ell)$ for all $1\leq \ell \leq m$.

Recall that if we let $N_H(T)$ be the normalizer of $T$ in $H$, then $N_H(T)/T$ is isomorphic to $W$. For each $\sigma \in W$ which is a permutation sending $(1, \ldots, d)$ to $(\sigma(1), \ldots, \sigma(d))$, we fix a representative of $\sigma$ in $H$ to be the $d\times d$ matrix $w_\sigma := (\delta_{i, \sigma(j)})_{1 \leq i, j \leq d} = (\delta_{\sigma^{-1}(i), j})$ where the notation $\delta_{x, y}=0$ if $x\neq y$, and $\delta_{x, y}=1$ if $x=y$.
Note that if we have another $d \times d$ matrix $A=(a_{k, l})$, we have the matrix multiplication:
$$ (\delta_{i, \sigma(j)}) (a_{k, l})  =(a_{\sigma^{-1}(k), l}), \quad
 (a_{k, l}) (\delta_{i, \sigma(j)})   = (a_{k, \sigma(l)}),$$
 and so in particular $\sigma^{-1} (a_{i,j}) \sigma = (a_{\sigma(i), \sigma(j)}).$

Let $C \subseteq \Rplus$ closed, we define the following subset of $W$:
$$W_C:=\{\sigma \in W : \sigma^{-1}(C)\subseteq R^+\}.$$

\begin{lemma} \cite[Lem. 2.3.6]{BH15} \label{LemmaBH}
With notations above, we have
$$W_C=\{\sigma \in W : w_{\sigma}^{-1}B_C w_{\sigma}  \subseteq B\}.$$
\end{lemma}
The above lemma says that conjugations of a matrix in $B_C$ by permutations in $W_C$ will stay upper triangular. In the following, we will sometimes simply use $\sigma$ to mean the matrix $w_{\sigma}$. If $?$ is a ring, we will use $B_C(?)$ to mean the subring of $\Mat_d(?)$ corresponding to the algebraic group $B_C$. That is, if $C= \{ \ee_{i_1}-\ee_{j_1}, \ldots, \ee_{i_m}-\ee_{j_m} \}$ is closed in $\Rplus$, then we let
$$B_C(?) : = \{ A=(a_{i, j}) \in \Mat_d(?) : A \textnormal{ is upper triangular, and }  a_{i, j}=0 \textnormal{ if } \ee_i-\ee_j \notin C    \}.$$ It is clear that for any $A \in B_C(?)$ and any $\sigma \in W_C$, $\sigma^{-1}A \sigma$ is an upper triangular matrix.

\section{Shape of upper triangular $(\varphi, \hat G)$-modules with $k_E$-coefficients} \label{section: Shape study}

In this section, we study the shape of upper triangular torsion $(\varphi, \hat G)$-modules, using results in \cite{Gao15unram}, as well as ideas in Section \ref{section: BH review}.

\subsection{Shape of $\varphi$}

\begin{prop} \label{varphishape}
With notations from (\textnormal{\textbf{CRYS}}). Suppose that $\overline \rho$ is upper triangular. Then $\barm \in \mathcal E(\barn_d, \ldots, \barn_1)$, where $\barn_i =\barm(t_i; a_i)$ for some $a_i \in k_E^{\times}$, and $\{t_1, \ldots, t_d\} =\{r_1, \ldots, r_d\}$ as sets.

Furthermore, there exists a basis $\bolde$ of $\barm$, such that the matrix $\Aphi$ of $\varphi$ with respect to this basis can be decomposed as $\Aphi =\widetilde \Aphi + u^p N$ where
\begin{enumerate}
  \item $\widetilde \Aphi$ is upper triangular, with diagonal equal to $[a_1u^{t_1}, \ldots, a_du^{t_d}]$, and $(\widetilde \Aphi)_{i, j} = u^{t_i} y_{i, j}$ for $i<j$ (here $(\widetilde \Aphi)_{i, j}$ is the element of $\widetilde \Aphi$ in the $(i, j)$-position), where
      \begin{itemize}
  \item $y_{ i, j}=0$ if $t_{j} < t_{i}$.
  \item $y_{ i, j} \in k_E$ if $t_{j}>t_{i}$.
      \end{itemize}
 \item $N \in \Mat_d(k_E[u])$ is strictly upper triangular (i.e., the diagonal is 0).
\end{enumerate}
\end{prop}

\begin{proof}
This is a slight generalization of \cite[Prop. 4.1]{Gao15unram} (using, in particular, \cite[Prop. 2.2, Prop. 2.3]{Gao15unram}). The novelty here is that we can allow the existence of nonzero morphisms $\barn_j \to \barn_i$ for some $j>i$, i.e., the situation in Statement (3) of \cite[Prop. 2.2]{Gao15unram} is allowed.

\textbf{Step 1}.
First of all, the existence and shape of $\barn_i$ is proved in \cite[Prop. 2.3]{Gao15unram}.
To construct the basis $\bolde$ and the upper triangular matrix $\Aphi$, we will apply \cite[Prop. 2.2]{Gao15unram}. For the convenience of the reader, let us give some more explanation of \textit{loc. cit.}, in the $\Qp$-case (the general unramified case is similar).
\footnote{The statement of \cite[Prop. 2.2]{Gao15unram} is correct. A minor imperfection is that in the proof of \textit{loc. cit.}, we cited \cite[Prop. 7.4]{GLS14}. Indeed, to be more precise, we should have cited  \cite[Prop. 5.1.3]{GLS15} instead (although as mentioned in \cite[Prop. 5.1.3]{GLS15}, their proof are almost identical). The difference between \cite[Prop. 7.4]{GLS14} and \cite[Prop. 5.1.3]{GLS15} is that in the latter situation, we can allow all the Hodge-Tate numbers to be nonzero (we thank one of the referees for pointing this out).}

The construction of the basis $\bolde$ (in \cite[Prop. 2.2]{Gao15unram}) when $d>2$ is really an easy inductive process from that of \cite[Prop. 5.1.3]{GLS15} (where $d=2$). Let us only sketch the case when $d=3$. That is, suppose now we have $\barm \in \mathcal E(\barn_3, \barn_2, \barn_1)$. So we have a basis $\{f_1, f_2, f_3\}$ such that
$$\varphi(f_1, f_2, f_3) = (f_1, f_2, f_3)
\begin{pmatrix}
a_1u^{t_1} & x & z \\
0 & a_2u^{t_2} & y \\
0 & 0 &  a_3u^{t_3}
\end{pmatrix}$$
The key point then is to make change of bases so that $x, y, z$ will satisfy the conditions in \cite[Prop. 2.2]{Gao15unram}. By \cite[Prop. 5.1.3]{GLS15} (the $d=2$ case), we can and do assume that $x$ already satisfies all the conditions in \cite[Prop. 2.2]{Gao15unram}. That is: $x$ is a polynomial in $k_E[u]$ of degree less than $t_2$, unless if there exists nonzero morphism $\barn_2 \to \barn_1$, then $x$ can have an extra term of degree $t_2+\frac{t_2-t_1}{p-1}$.

The next step is to alter $f_3$ in order to make $y, z$ satisfy \cite[Prop. 2.2]{Gao15unram}. We can first change $f_3$ to $f_3' =f_3 + \alpha f_2$ as in the proof of \cite[Prop. 5.1.3]{GLS15} to make $y$ to some $y'$ that satisfy \cite[Prop. 2.2]{Gao15unram}. Note that this process will not have any effect on $x$, but it will alter $z$.
So now we are in the situation
$$\varphi(f_1, f_2, f_3') = (f_1, f_2, f_3')
\begin{pmatrix}
a_1u^{t_1} & x & z' \\
0 & a_2u^{t_2} & y' \\
0 & 0 & 0 a_3u^{t_3}
\end{pmatrix}$$
where both $x, y'$ satisfy \cite[Prop. 2.2]{Gao15unram}. Now we only need to change $f_3'$ to some $f_3''=f_3'+\beta f_1$ in order to make $z'$ satisfy \cite[Prop. 2.2]{Gao15unram}. Note that there is no extension between $\barn_3$ and $\barn_1$, so we can not directly apply \cite[Prop. 5.1.3]{GLS15} to get $f_3''$. However, the ``$f_2$-parts" of $\varphi(f_3'') =\varphi(f_3') +\varphi(\beta)a_1u^{t_1} f_1$ and $\varphi(f_3')$ are the same. So we can ``forget" about $f_2$ and pretend that there is an extension between $\barn_3$ and $\barn_1$. The same process as in \cite[Prop. 5.1.3]{GLS15} will in the end produce our desired basis $\bolde$.

\textbf{Step 2}.
Now let us discuss about the ``extra terms".
Recall that in \cite[Prop. 2.2]{Gao15unram}, when there exists nonzero morphisms $\barn_j \to \barn_i$ for some $j>i$, then $\Aphi$ can have \emph{extra} terms as described in Statement (3) of \textit{loc. cit.}, and this extra term has degree $t_j+\frac{t_j -t_i}{p-1}$. Note that in order to have $\barn_j \to \barn_i$ for $j>i$, the only possibility is to have $t_j-t_i =p-1$ and $a_i=a_j$ (easy by \cite[Lem. 1.13]{Gao15unram} since we are in the $\Qp$ situation). So the extra terms are always of degree $p$ or $p+1$, i.e., the extra terms are always divisible by $u^p$. (In fact, clearly we can only have at most two extra terms). Decompose $\Aphi$ as $\widetilde \Aphi + u^p N$ where $u^pN$ are the extra terms.

\textbf{Step 3}.
Finally, we only need to prove the properties regarding $y_{i, j}$.
We argue similarly as in \cite[Prop. 2.3]{Gao15unram}, let $\bolde'$ be another basis of $\barm$ such that $\varphi(\bolde') = \bolde' X [u^{r_1}, \ldots, u^{r_d}]$ where $X \in \Mat_d(k_E\llb u\rrb)$ as in \cite[Thm. 2.1]{Gao15unram}.
Let $\bolde'=\bolde T$ for some matrix $T \in \GL_d(k_E\llb u\rrb)$, then $\Aphi = TX [u^{r_1}, \ldots, u^{r_d}] \varphi(T^{-1})$.
Similarly as in \cite[Prop. 4.1]{Gao15unram}, let $\varphi(T)=P+u^pQ$ for some $P \in \GL_d(k_E), Q \in \Mat_d(k_E\llb u\rrb)$, and let $R \in \GL_d(k_E)$ such that $R^{-1} [u^{r_1}, \ldots, u^{r_d}] R = [u^{t_1}, \ldots, u^{t_d}]$, then we have
$$(\widetilde \Aphi + u^p N)(P+u^pQ)R = TXR[u^{t_1}, \ldots, u^{t_d}].$$
So we have $u^{t_i} \mid \col_i(\widetilde{\Aphi} PR) $.
Now we can again apply \cite[Lem. 4.3]{Gao15unram} to conclude (note that $\widetilde{\Aphi}$ satisfies property (DEG) of \textit{loc. cit.}, since we removed the extra terms $u^pN$ from $\Aphi$).
\end{proof}

With notations in Proposition \ref{varphishape}, we can define the following subset $C$ of $R^+$:
\numequation \label{BC}
  C :=\{\ee_i -\ee_j: i<j, t_i<t_j\}.
  \end{equation}
It is easy to see that $C$ is closed in $R^+$, and $\widetilde \Aphi$ is a matrix in the subring $B_C(k_E[u])$. But in fact, we also have $\Aphi \in B_C(k_E[u])$, because the extra terms in $u^pN$ only show up in positions $(i, j)$ where $t_i<t_j$.

\begin{prop} \label{varphiBC}
There exists a unique $\sigma \in W_C$ such that $\sigma^{-1} \Aphi \sigma$ is still upper triangular, and $\diag(\sigma^{-1} \Aphi \sigma) = [a_{\sigma(1)}u^{r_1}, \ldots, a_{\sigma(d)}u^{r_d}]$.
\end{prop}
\begin{proof}
The uniqueness of $\sigma$ is determined since we have $t_{\sigma(i)}=r_i, \forall i$, that is,
\numequation \label{order}
  t_{\sigma(1)}< \ldots < t_{\sigma(d)}.
  \end{equation}
It suffices to show that $\sigma \in W_C$ ($\Leftrightarrow \sigma^{-1}(W_C) \subseteq R^{+}$), i.e., if $\ee_i -\ee_j \in C$, then $\sigma^{-1}(i) <\sigma^{-1}(j)$.
Let $x = \sigma^{-1}(i)$ and $\sigma^{-1} (j) =y$.
Then $t_i=t_{\sigma(x)}< t_{\sigma(j)} =t_j$.
So by \eqref{order}, we must have $x<y$.
\end{proof}

\begin{remark}The following remark is suggested by one of the referees.
Since we have already shown that $(A_{\varphi})_{i, j}=0$ when $t_j < t_i$ (where $j>i$), we could use an elementary ``swapping" process to obtain the above Proposition. Namely, suppose for example $t_{i+1}<t_i$, we could simply change the basis $(e_1, \ldots, e_i, e_{i+1}, \ldots, e_d)$ to $(e_1, \ldots, e_{i+1}, e_{i}, \ldots, e_d)$; the matrix for $\varphi$ will remain upper triangular. After all these possible two by two swappings, the $u$-power on the diagonal will become eventually increasing.

As the readers can see, this elementary swapping process is precisely the key idea in the Breuil-Herzig group theory that we reviewed in Section \ref{section: BH review}. Indeed, the ``ordinary part" of the $p$-adic Langlands in \cite{BH15} is precisely built out from $\GL_2$! It is also interesting to point out in the paper \cite{GS16}, a similar Weyl group element played a similar useful role in determining the locally algebraic vectors in the ``ordinary part" of \cite{BH15} (see the remarks following \cite[Thm 1.2]{GS16}).

We have chosen to keep the Breuil-Herzig theory in our paper (instead of the more elementary swapping process), because it does make the argument cleaner. Also, as we mentioned in the Introduction, this indeed provides a natural example of the Breuil-Herzig group theory.
\end{remark}

\subsection{Shape of $\tau$}

Our following lemma (Lemma \ref{zerosolution}) is valid for any $K/\Qp$. So we use notations introduced in Section \ref{section intro}. Recall that $u = (\pi_{n})_{n=0}^{\infty} \in R$, and we normalize the valuation on $R$ so that $v_R(u)=\frac{1}{e}$ where $e$ is the ramification degree of $K/\Qp$.
For $\zeta \in R\otimes_{\Fp}k_E$, write it as $\zeta= \sum_{i=1}^m y_i\otimes a_i$ where $y_i \in R$, and $a_i \in k_E$ are independent over $\Fp$. Let $$v_R(\zeta):=\textnormal{min}\{v_R(y_i)\}.$$
Then by \cite[Lem. 5.6]{Gao15unram}, $v_R$ is a well-defined valuation on $R\otimes_{\Fp}k_E $ (so in particular, it does not depend on the sum representing $\zeta$). In particular, $v_R(\varphi(\zeta))=pv_R(\zeta)$. We also use the convention that $v_R(0)=+\infty$.

\begin{lemma} \label{zerosolution}
Let $\zeta \in R\otimes_{\Fp}k_E$ with $v_R(\zeta)>0$ , such that
$$\zeta \tau(\varphi(u^b))  = \varphi(u^a) \varphi(\zeta)$$
 for some $a>b\geq 0$, then $\zeta=0$.
\end{lemma}
\begin{proof}
Note that $\tau(u)=u\underline \epsilon$, where $\underline \epsilon=(\mu_{p^n})_{n=0}^{\infty} \in R$ . Consider the valuation on both side of the equation, then $v_R(\zeta) +\frac{pb}{e}=\frac{pa}{e} +pv_R(\zeta)$. The only possibility is when $v_R(\zeta)=+\infty$.
\end{proof}

Now we return to the $\Qp$ case.
\begin{prop} \label{tauBC}
With notations in Proposition \ref{varphishape}, let $\Atau \in \Mat(R\otimes_{\Fp}k_E)$ is the matrix of $\tau$ with respect to the basis $1\otimes_{\varphi}\bolde$. Then $\Atau$ is in the subring $B_C(R\otimes_{\Fp}k_E)$ defined by \eqref{BC}, i.e., if $i<j$ and $t_i>t_j$, then $(\Atau)_{i, j}=0$.
\end{prop}
\begin{proof}
This is easy consequence of the following lemma. Note that for any $i<j$, $v_R((\Atau)_{i, j})>0$ by \cite[Lem. 5.7]{Gao15unram}.
\end{proof}

\begin{lemma}
Let $F=(f_{i, j}) \in \Mat_d(k_E\llb u \rrb), M=(m_{i,j }) \in \Mat_d(R\otimes_{\Fp}k_E)$ two upper triangular matrices. Suppose $\diag(F) =[a_1u^{t_1}, \ldots, a_du^{t_d}]$ where $a_i \in k_E^{\times}$ and $t_i$ are \emph{distinct} non-negative integers. Suppose that
\begin{itemize}
  \item If $i<j$ and $t_i>t_j$, then $f_{i, j}=0$,
  \item $v_R(m_{i,j})>0, \forall i<j$, and
  \item $M \tau(\varphi(F)) = \varphi(F)\varphi(M)$.
\end{itemize}
Then $M \in B_C(R\otimes_{\Fp}k_E)$, where $C :=\{\ee_i -\ee_j: i<j, t_i<t_j\}$ the closed subset of $R^{+}$.
\end{lemma}
\begin{proof}
We prove by induction on the dimension $d$. When $d=1$, there is nothing to prove. Suppose the lemma is true for dimension less than $d$, and consider it for $d$.

We can apply the induction hypothesis to $F_{1, 1}$ and $M_{1, 1}$ (resp. $F_{d, d}$ and $M_{d, d}$), where $F_{1, 1}$ is the co-matrix of $F$ by deleting the 1st row and 1st column (and similarly for $M_{1, 1}, F_{d, d}$ and $M_{d, d}$).
So we only need to deal with the element on the most upper right corner. That is, we only need to prove that if $t_1 >t_d$, then $m_{1, d}=0$.

For any $2 \leq i \leq d$, we have
\begin{itemize}
  \item (Case 1) If $t_i>t_1>t_d$, then $f_{i, d}=0$ (property of $F$), and $m_{i, d}=0$ (induction hypothesis).
  \item (Case 2) If $t_1 >t_i$, then $f_{1, i}=0$ (property of $F$), and $m_{1, i}=0$ (induction hypothesis).
\end{itemize}

By the condition that $M \tau(\varphi(F)) = \varphi(F)\varphi(M)$, we must have
$$  \sum_{i=1}^d  m_{1, i} \tau(\varphi (f_{i, d}))  = \sum_{i=1}^d \varphi(f_{1, i}) \varphi(m_{i, d}).$$
So we will always have $m_{1, d} \tau(\varphi(u^{t_d}))  = \varphi(u^{t_1}) \varphi(m_{1, d})$, because all the other terms vanish.
Now we can conclude $m_{1, d}=0$ by Lemma \ref{zerosolution}.
\end{proof}

\section{Crystalline lifting theorem} \label{section: lifting}

\begin{thm}
With notations in (\textnormal{\textbf{CRYS}}), and suppose that $\barrho$ is upper triangular. Suppose $\barrho \in \mathcal E(\barchi_1, \ldots, \barchi_d)$ such that $\barchi_i \barchi_j^{-1} \neq \overline{\varepsilon}_p, \forall i \neq j$.
Then there exists an upper triangular crystalline representation $\rho'$ such that $\barrho' \cong \barrho$, and $\HT(\rho')=\HT(\rho)$ as sets.
\end{thm}

\begin{proof}
Recall that $\bolde=(e_1, \ldots, e_d)$ is the basis of $\barm$ in Proposition \ref{varphishape}.
Let $\sigma \in W_C$ be the unique element as in Proposition \ref{varphiBC},
and denote $\bolde^{\sigma}:=(e_{\sigma(1)}, \ldots, e_{\sigma(d)})$.
By \textit{loc. cit.}, the matrix of $\varphi$ for $\barm$ with respect to the basis $\bolde^{\sigma}$ (which is $\sigma^{-1} \Aphi \sigma$) is still upper triangular. By Proposition \ref{tauBC} and Lemma \ref{LemmaBH}, the matrix of $\tau$ for $\barhatm$ with respect to the basis $1\otimes_{\varphi}\bolde^{\sigma}$ (which is $\sigma^{-1} \Atau \sigma$) is also upper triangular.
That is to say (by Lemma \ref{pnot2}), $\barhatm \in \mathcal E(\barhatn_{\sigma(d)}, \ldots, \barhatn_{\sigma(1)})$, where $\barhatn_{\sigma(i)}: = \barhatm( r_i; a_{\sigma(i)})$.
And so $\barrho=\hat T(\barhatm) \in  \mathcal E(\barchi_{\sigma(1)}, \ldots, \barchi_{\sigma(d)})$.

By Lemma \ref{lemma: rank 1}(3), each $\barchi_{\sigma(i)}$ has a crystalline lift $\chi_{\sigma(i)}=:\hat T(\hatm(r_i; \hat{a}_{\sigma(i)}))$, where $ \hat{a}_{\sigma(i)} \in \O_E^{\times}$ is any lift of $a_{\sigma(i)}$.
Since $r_1<\ldots <r_d$, by \cite[Lem. 3.1.5]{GG12} (note that our convention of Hodge-Tate weights is the opposite of \textit{loc. cit.}), $\barrho$ has an upper triangular crystalline lift $\rho'$ such that $\rho'\in \mathcal E(\chi_{\sigma(1)}, \ldots, \chi_{\sigma(d)})$. Let us remark here that $\rho'$ is in fact \emph{ordinary} in the sense of \cite[Def. 3.1.3]{GG12}.
\end{proof}

\section*{Acknowledgements}
The author would like to heartily thank Florian Herzig for many useful and inspiring discussions, and for patiently answering many questions. The author also would like to thank the anonymous referee(s) for useful comments which help to clarify the exposition.
This paper is written when the author is a postdoc in Beijing International Center for Mathematical Research and University of Helsinki, and we would like to thank the institutes for the hospitality.
This work is partially supported by China Postdoctoral Science Foundation General Financial Grant 2014M550539.


\end{document}